\documentclass[a4paper,11pt,twoside]{amsart}

\usepackage[latin1]{inputenc}
\usepackage{amsmath, amsfonts, amssymb,amsthm}
\usepackage[mathscr]{eucal} %Proporciona el comando
\usepackage[pdftex]{graphicx}
\usepackage{color}

\usepackage[T1]{fontenc}
\usepackage[sc]{mathpazo}
\usepackage{hyperref}

\usepackage{enumerate}

\definecolor{alert}{rgb}{0.8,0,0}
\newcommand{\alert}[1]{\textbf{\textcolor{alert}{#1}}}

\newcommand{\alertm}[1]{%
	\marginpar{%
		\ifodd\value{page} \raggedright \else \raggedleft \fi
		\footnotesize{\alert{#1}}
	}
}

\newcommand{\N}{\mathbb{N}}

\newcommand{\R}{\mathbb{R}}

\newcommand{\s}{\mathbb{S}}
\newcommand{\h}{\mathbb{H}}
\newcommand{\E}{\mathbb{E}}
\newcommand{\M}{\mathbb{M}}

\newcommand{\Nil}{\mathrm{Nil}}

\newcommand{\df}{\,\mathrm{d}}

\newcommand{\sm}{\smallsetminus}

\newcommand{\p}{\partial}

\newtheorem{theorem}{Theorem}
\newtheorem{proposition}[theorem]{Proposition}
\newtheorem{corollary}[theorem]{Corollary}
\newtheorem{lemma}[theorem]{Lemma}

\theoremstyle{definition}
  \newtheorem{definition}[theorem]{Definition}

\theoremstyle{remark}
\newtheorem{remark}[theorem]{Remark}

\numberwithin{equation}{section}

\title{On complete constant mean curvature vertical multigraphs in $\E(\kappa,\tau)$}

\author{Jos\'e M. Manzano}
\author{M. Magdalena Rodr\'\i guez}
\address{Departamento de Geometr\'{\i}a  y Topolog\'{\i}a \\
Universidad de Granada \\
18071 Granada, Spain} 
\email{jmmanzano@ugr.es, magdarp@ugr.es. }

\thanks{Research partially supported by the MCyT-Feder research
  project MTM2011-22547, the Regional J. Andaluc\'\i a Grant
  no. P09-FQM-5088 and the CEI BioTIC GENIL project (CEB09-0010)
  no. PYR-2010-21}



\begin{document}

\begin{abstract}
  We prove that any complete surface with constant mean curvature in a
  homogeneous space $\E(\kappa,\tau)$ which is transversal to the
  vertical Killing vector field is, in fact, a vertical graph.  As a
  consequence we get that any orientable, parabolic, complete,
  immersed surface with constant mean curvature $H$ in
  $\E(\kappa,\tau)$ (different from a horizontal slice in
  $\s^2\times\R$) is either a vertical cylinder or a vertical graph
  (in both cases, it must be $4H^2+\kappa\leq0$).
\end{abstract}

\maketitle

\section{Introduction}

The classical theory of constant mean curvature surfaces in
$3$-manifolds is still an active field of research nowadays. Among all
ambient $3$-manifolds, the most studied subclass consists of those
simply-connected with constant sectional curvature, the so-called
\emph{space forms} ($\R^3$, $\s^3$ or $\h^3$). Their isometry group
has dimension $6$ and acts transitively on their tangent bundle.
Apart from the space forms, the most symmetric ones are those
simply-connected whose isometry group has dimension $4$. They are
usually denoted by $\E(\kappa,\tau)$, where $\kappa,\tau\in\R$ satisfy
$\kappa-4\tau^2\neq0$, and include the Riemannian product manifolds
$\s^2(\kappa)\times\R$ and $\h^2(\kappa)\times\R$ (when $\tau=0$), the
Heisenberg space Nil$_3$ (for $\kappa=0$), the universal cover
$\widetilde{\mbox{SL}_2}(\R)$ of the special linear group SL$_2(\R)$
endowed with some special left-invariant metrics (for $\kappa<0$ and
$\tau \neq 0$), and the Berger spheres (for $\kappa>0$ and $\tau \neq
0$). See~\cite{Daniel07} for more details.

The spaces $\E(\kappa,\tau)$ are characterized by admitting a
Riemannian submersion $\pi:\E(\kappa,\tau)\to\M^2(\kappa)$ with
constant bundle curvature $\tau$, where $\M^2(\kappa)$ stands for the
simply-connected Riemannian surface with constant curvature $\kappa$,
such that the fibers of $\pi$ are the integral curves of a unit
Killing vector field in $\E(\kappa,\tau)$ (see~\cite{man12}). In what
follows, we will refer to this field as the \emph{vertical} Killing
vector field and it will be denoted by $\xi$.

In the theory of constant mean curvature surfaces ($H$-surfaces in the
sequel) in $\E(\kappa,\tau)$, vertical multigraphs play an
important role. A surface~$\Sigma$ immersed in $\E(\kappa,\tau)$ is
said to be a \emph{vertical multigraph} if the following equivalent
conditions hold:
\begin{enumerate}
\item[a)] $\Sigma$ is transversal to the vertical Killing vector field
  $\xi$.
\item[b)] The angle function $\nu=\langle N,\xi\rangle$ has no zeros in
  $\Sigma$, where $N$ is a global unit normal vector field of
  $\Sigma$.
\item[c)] The projection $\pi_{|\Sigma}:\Sigma\to\M^2(\kappa)$ is a local
  diffeomorphism.
\end{enumerate}
The first two conditions are trivially equivalent, whereas the third
one follows from the fact that the absolute value of the Jacobian of
$\pi_{|\Sigma}$ equals~$|\nu|$.

The need of knowing whether a vertical $H$-multigraph is or not
embedded often arises in the study of such surfaces. In this paper we
solve this problem by proving that they are always \emph{vertical
  graphs} (i.e., they intersect at most once each integral curve of
$\xi$), and in particular embedded, when we additionally assume that
the vertical $H$-multigraph is complete. (We observe that the
completeness asumption is needed: A counterexample is given, for
example, by the helicoid of $\h^2\times\R$ constructed by Nelli and
Rosenberg in~\cite{NR02} minus a neighborhood of its axis.) In this line, we recall the following known results:
\begin{itemize}
\item Hauswirth, Rosenberg and Spruck proved a half-space theorem for
  properly embedded $\frac{1}{2}$-surfaces in $\h^2\times\R$ and
  concluded that a complete vertical $\frac{1}{2}$-multigraph in
  $\h^2\times\R$ is always an entire vertical graph
  (see~\cite{HRS}), i.e., it intersects exactly once each integral curve
  of~$\xi$. Later on, by using another halfspace theorem for properly
  immersed surfaces in $\Nil_3$, Daniel and Hauswirth~\cite{dh1} proved the same result for minimal vertical multigraphs
  in~$\Nil_3$. Finally, using this theorem by Daniel and Hauswirth,
  the classification theorem by Fern\'andez and Mira in~\cite{fm1} and
  the Daniel correspondence~\cite{Daniel07}, it is possible to extend
  the aforementioned results to prove that a complete vertical
  $H$-multigraph in $\E(\kappa,\tau)$ satisfying $4H^2+\kappa=0$ is an
  entire vertical graph (see also Corollary~4.6.3 in~\cite{kias} for a
  complete reference).
\item Espinar and Rosenberg proved in \cite{ER09} that there are no
  complete vertical H-multigraphs in $\h^2\times\R$ for
  $H>\frac{1}{2}$.  In a joint work with Joaqu\'\i n P\'erez (see the
  proof of Theorem 2 in~\cite{MaPR}), the authors proved using a
  different approach that the only complete vertical $H$-multigraphs
  in $\E(\kappa,\tau)$ with $4H^2+\kappa>0$ are the horizontal slices
  $\s^2(\kappa)\times\{t_0\}$ in $\s^2(\kappa)\times\R$, for any
  $\kappa>0$ (and $\tau=H=0$). The latter result has also been proved in~\cite{P}.
\end{itemize}

To obtain the general result stated below, it only remains to
study the case $4H^2+\kappa<0$. We highlight that the geometry of an $H$-surface in the homogeneous
spaces $\E(\kappa,\tau)$ varies essentially depending on the sign of
$4H^2+\kappa$. For instance, it is known that constant mean curvature
spheres exist if, and only if, $4H^2+\kappa>0$; see Theorem 2.5.3
in~\cite{kias}.

\begin{theorem}\label{teor:grafos-multigrafos}
  Let $\Sigma$ be a complete vertical $H$-multigraph in
  $\E(\kappa,\tau)$. Then, one of the following statements hold:
  \begin{enumerate}
  \item[a)] $\E(\kappa,\tau)=\s^2(\kappa)\times\R$, $H=0$ and $\Sigma=\s^2(\kappa)\times\{t_0\}$, for some $t_0\in\R$.
  \item[b)] $4H^2+\kappa\leq 0$ and $\Sigma$ is a vertical
    graph. Moreover, if the equality holds then the graph is entire.
  \end{enumerate}
\end{theorem}

We remark that the condition $4H^2+\kappa=0$ in
Theorem~\ref{teor:grafos-multigrafos} is not necessary in order to
obtain entire vertical $H$-graphs: Besides horizontal slices, other
entire minimal vertical graphs in $\h^2\times\R$ have been constructed
by Nelli and Rosenberg in~\cite{NR02}, by Collin and Rosenberg in~\cite{CR10} and by Mazet, Rosenberg and the second author
in~\cite{MRR11}. The reader can also find some examples of
rotationally invariant entire vertical $H$-graphs in $\h^2\times\R$
for any $0<H\leq \frac12$ in~\cite{NR06}.
We also emphasize that there exist many complete vertical $H$-graphs
in $\h^2\times\R$ which are not entire:
\begin{itemize}
\item On the one hand, Sa Earp~\cite{SE} and Abresch gave a explicit
  complete minimal vertical graph defined on half a hyperbolic plane.
  In~\cite{CR10,MRR11,Rod}, some complete minimal examples in $\h^2\times\R$
  (which are vertical graphs over simply connected domains boun\-ded by
  --possibly infinitely many-- complete geodesics where the graph has
  non-bounded boundary data, and/or finitely many arcs at the infinite
  boundary of $\h^2$) are given. Melo constructed in~\cite{Melo}
  minimal examples in $\widetilde{\mbox{SL}_2}(\R)$ similar to those
  in~\cite{CR10}.
\item On the other hand, Folha and Melo obtained
  in~\cite{FoM10} vertical $H$-graphs in $\h^2\times\R$, with $0<H<\frac12$,
  over simply-connected domains bounded by an even number of curves of
  geodesic curvature $\pm 2H$ (disposed alternately) over which the
  graphs go to $\pm\infty$.
\end{itemize}

Let us now explain a consequence of
Theorem~\ref{teor:grafos-multigrafos}.  We consider the {\it stability
  operator} of a two-sided $H$-surface immersed in $\E(\kappa,\tau)$,
given by
\begin{equation}\label{eq:stab}
  L=\Delta+|A|^2+\mbox{Ric}(N),
\end{equation}
where $\Delta$ stands for the Laplacian with respect to the induced
metric on $\Sigma$, and $A$ and $N$ denote respectively the shape
operator and a unit normal vector field of $\Sigma$.  The surface
$\Sigma$ is said to be {\it stable} when $-L$ is a non-negative
operator (see~\cite{mpr19} for more details on stable $H$-surfaces).
Rosenberg proved in~\cite{R06} the non-existence of stable
$H$-surfaces in $\E(\kappa,\tau)$ provided that $3H^2+\kappa>\tau^2$,
other than $\s^2(\kappa)\times\{t_0\}$ in $\s^2(\kappa)\times\R$.  It
is conjectured that the optimal condition for such non-existence
result is $4H^2+\kappa>0$.  In a joint work with P\'erez~\cite{MaPR},
the authors slightly improve Rosenberg's bound in the general case and
obtain the expected bound under the additional assumption of
parabolicity.
We recall that a Riemannnian manifold $\Sigma$ is said to be {\it
  parabolic} when the only positive superharmonic functions defined on
$M$ are the constant functions.
As a direct consequence of Theorem~\ref{teor:grafos-multigrafos} above
and Theorem 2 in~\cite{MaPR}, we get the following nice classification
result:

\begin{corollary}\label{cor:stable}
  Let $\Sigma$ be an orientable, parabolic, complete, stable
  $H$-surface in $\E(\kappa,\tau)$.  Then, one of the following
  statements hold:
  \begin{enumerate}
  \item[a)] $\E(\kappa,\tau)=\s^2(\kappa)\times\R$, $H=0$ and $\Sigma=\s^2(\kappa)\times\{t_0\}$, for some $t_0\in\R$.
  \item[b)] $4H^2+\kappa\leq 0$ and $\Sigma$ is either a
    vertical graph or a vertical cylinder over a complete curve of
    geodesic curvature $2H$ in $\M^2(\kappa)$.
  \end{enumerate}
\end{corollary}

\section{Preliminaries}

From now on, $\Sigma$ will denote a complete vertical $H$-multigraph
in $\E(\kappa,\tau)$, with $4H^2+\kappa<0$. Let us remark that the
latter condition implies $\kappa<0$, so $\E(\kappa,\tau)$ admits a
fibration over $\h^2(\kappa)$. By applying a convenient homothety in
the metric, there is no loss of generality in supposing that
$\kappa=-1$. Hence, we can consider the disk $\mathbb
D(2)=\{(x,y)\in\R^2:x^2+y^2<4\}$ and the model $\E(-1,\tau)=\mathbb
D(2)\times\R$, endowed with the Riemannian metric
\[
\df s^2=\lambda^2(\df x^2+\df y^2)+(\df z+\tau\lambda (y\df x-x\df
y))^2,
\] 
where $\lambda:\mathbb D(2)\to\R$ is given by
$\lambda(x,y)=\bigl(1-\tfrac{1}{4}(x^2+y^2)\bigr)^{-1}$. In this
model, the Riemannian fibration is nothing but
$\pi:\E(-1,\tau)\to\h^2$, $\pi(x,y,z)=(x,y)$, when we identify
$\h^2\equiv(\mathbb D(2),\lambda^2(\df x^2+\df y)^2)$.

Given an open set $\Omega\subset\h^2$ and a function $u\in
C^\infty(\Omega)$, the graph associated to $u$ is defined as the
surface parametrized by
\[
F_u:\Omega\to\E(-1,\tau),\quad F_u(x,y)=(x,y,u(x,y)).
\]

Let us also fix the following notation, for any $R>0$: 
\begin{itemize}
\item Given $x_0\in\h^2$, we denote by $B(x_0,R)$ the ball in $\h^2$
  of center $x_0$ and radius $R$.
\item Given $p_0\in\Sigma$, we denote by $B_\Sigma(p_0,R)$ the
  intrinsic ball in $\Sigma$ of center $p_0$ and radius $R$.
\end{itemize}

The proof of Theorem~\ref{teor:grafos-multigrafos} relies on some
technical results given originally by Hauswirth, Rosenberg and Spruck
in~\cite{HRS}. Although they treat the case $H=\frac{1}{2}$ in
$\h^2\times\R$, their arguments can be directly generalized to
$H$-surfaces in $\E(\kappa,\tau)$ with $4H^2+\kappa\leq 0$, giving
rise to Lemma~\ref{lema:asintotico2}
below.

Take $x_0\in\h^2$ and $R>0$ such that there exists $u\in
C^\infty(B(x_0,R))$ with $F_u(B(x_0,R))\subset\Sigma$.
(As $\pi_{|\Sigma}$ is a local diffeomorphism, this can be done for
any $x_0\in\pi(\Sigma)$.)  We also fix a unit normal vector field $N$
of $\Sigma$ so that the angle function $\nu=\langle N,\xi\rangle$ is
positive. 
We observe that $\nu$ lies in the kernel of the stability operator of
$\Sigma$, defined in~\eqref{eq:stab}, see~\cite{bce1}. Since $\nu$ has
no zeros, we deduce by a theorem given by Fischer-Colbrie~\cite{fi1}
(see Lemma 2.1 in~\cite{mpr19}) that $\Sigma$ stable. This is an
important fact in the proof of Lemma~\ref{lema:asintotico2}.

Suppose that $\hat x\in\partial B(x_0,R)$ satisfies that $u$ cannot be
extended to any neighborhood of $\hat x$ in $\h^2$ as a vertical
$H$-graph. Given a sequence $\{x_n\}\subset B(x_0,R)$ converging to
$\hat x$ and calling $p_n=F_u(x_n)$, the arguments in~\cite{HRS} can
be extended to show that the sequence of surfaces $\{\Sigma_n\}$
(where $\Sigma_n$ results from translating $\Sigma$ vertically so that
$p_n$ is at height zero), converges to a vertical cylinder
$\pi^{-1}(\Gamma)$, for a curve $\Gamma\subset\h^2$ of constant
geodesic curvature $2H$ or $-2H$ which is tangent to $\partial
B(x_0,R)$ at $\hat x$. Note that the condition $4H^2-1<0$ implies that
$\Gamma$ is non-compact.

Given $\delta>0$, we denote by $N_\delta(\Gamma)$ the open
neighborhood of $\Gamma$ in $\h^2$ of radius $\delta$.
Moreover, given $x\in\h^2\sm\Gamma$, we call
$N_\delta(\Gamma,x)=N_\delta(\Gamma)\cap U$, where $U$ is the
connected component of $\h^2\sm\Gamma$ containing $x$. By coherence,
we denote $U=N_\infty(\Gamma,x)$. 

\begin{lemma}[\cite{HRS}]\label{lema:asintotico2}
  In the setting above, suppose that $\hat x\in\partial B(x_0,R)$
  satisfies that $u$ cannot be extended to any neighborhood of $\hat
  x$ in $\h^2$ as a vertical $H$-graph.  Then, there exist $\delta>0$
  and a complete curve $\Gamma\subset\h^2$ of constant geodesic
  curvature $2H$ or $-2H$, tangent to $\partial B(x_0,R)$ at $\hat x$,
  such that $u$ extends as a vertical $H$-graph to $B(x_0,R)\cup
  N_\delta(\Gamma,x_0)$ with infinite constant boundary values along
  $\Gamma$.
\end{lemma}

\begin{remark}\label{obse:signo-infinito}
  In our study, we will always find a dichotomy between geodesic
  curvature $2H$ or $-2H$, as well as boundary values $+\infty$ or
  $-\infty$. Although, the arguments below do not depend on these
  signs, it is worth saying something about the relation between them
  in order to describe the complete vertical graphs. As mentioned
  above, it can be shown that the $H$-multigraphs $\Sigma_n$ converge
  uniformly on compact subsets to a vertical cylinder and we are
  considering the unit normal vector field of $\Sigma$ which points
  upwards. By analyzing the normals of $\Gamma$ in $\h^2$ and of
  $\pi^{-1}(\Gamma)$ in $\E(-1,\tau)$, it is not difficult to realize
  that if $u$ tends to $+\infty$ (resp.~$-\infty$) along $\Gamma$, the
  geodesic curvature of $\Gamma$ will be $2H$ (resp.~$-2H$) with
  respect to the normal vector pointing to the domain of definition of
  the graph.  This fact was proved for minimal surfaces in
  $\h^2\times\R$ by Nelli and Rosenberg in~\cite{NR02}; for
  $H$-surfaces in $\h^2\times\R$, for $0<H\leq 1/2$, by Hauswirth,
  Rosenberg and Spruck in~\cite{HRS09}; and for minimal surfaces in
  $\widetilde{\mathrm{SL}_2}(\R)$ by Younes in~\cite{Younes}.
\end{remark}

\section{The proof of Theorem~\ref{teor:grafos-multigrafos}}

As explained above, supposed that $\Sigma$ is a complete vertical
$H$-multigraph in $\E(-1,\tau)$ with $0\leq H<\frac12$, we will show
that $\Sigma$ is a vertical graph. The proof relies on the following
result.

\begin{proposition}\label{prop}
  Let $\Sigma$ be a complete vertical $H$-multigraph in
  $\E(-1,\tau)$. Given $p\in\Sigma$ and $R>0$, there exist an open set
  $\Omega\subseteq\h^2$ and a function $u\in C^\infty(\Omega)$ such
  that $B_\Sigma(p,R)\subseteq F_u(\Omega)$.
\end{proposition}

Note that the fact that $\Sigma$ is indeed a graph follows from
Proposition~\ref{prop}: Reasoning by contradiction, if there existed
$p,q\in\Sigma$, $p\neq q$, which project by $\pi:\E(-1,\tau)\to\h^2$
on the same point, then we could take $R$ bigger than the intrinsic
distance from $p$ to $q$ to reach a contradiction. Thus, the rest of
this section will be devoted to prove Proposition~\ref{prop}.

From now on, we fix a point $p\in\Sigma$ and denote $x_0=\pi(p)$. The
following definition will be useful in the sequel.

\begin{definition}
  We say that $\Omega\subset\h^2$ is {\em admissible} if it is a
  connected open set containing $x_0$ for which there exists $u\in
  C^\infty(\Omega)$ such that $F_u(\Omega)\subset\Sigma$ and
  $F_u(x_0)=p$.
\end{definition}

If $\Omega\subset\h^2$ is admissible, then the function $u$ in the
definition above is unique, and we will call it the function
\emph{associated} to $\Omega$. The technique we will use to prove
Proposition~\ref{prop} consists of enlarging gradually an initial
admissible domain until it eventually contains the projection of an
arbitrarily large intrinsic ball.

As $\nu>0$, there exists a neighborhood of $p$ in $\Sigma$ which
projects one-to-one to a ball $B(x_0,\rho)$, for some $\rho>0$.  In
other words, there exists $\rho>0$ such that $B(x_0,\rho)$ is
admissible. Let us consider
\begin{equation}\label{eqn:multigrafo-radio-inicial}
  R_0=\sup\{\rho>0:B(x_0,\rho)\text{ is admissible}\}.
\end{equation}
If $R_0=+\infty$, Proposition~\ref{prop} follows trivially for such a
point $p$, so we assume $R_0<+\infty$. This implies that $B(x_0,R_0)$
is a \emph{maximal admissible ball}, and it will play the role of our
initial domain.

\begin{lemma}\label{lem1}
  Let $\Omega$ an admissible domain such that $\partial\Omega$ is a
  piecewise $C^2$-embedded curve. Suppose that there exists $\hat x$
  in a regular arc of $\partial\Omega$ such that $\Omega$ cannot be
  extended as an admissible domain to any neighborhood of $\hat
  x$. Then:
  \begin{itemize}
  \item[i)] The geodesic curvature of $\partial\Omega$ at $\hat
  x$ with respect to its inner normal vector is at least $-2H$.
\item[ii)] There exists a curve $\Gamma\subset\h^2$ with constant
  geodesic curvature $2H$ or $-2H$, tangent to $\partial\Omega$ at
  $\hat x$, and there exist $\delta>0$, $y\in\Omega$ and a function
  $v\in C^\infty(N_\delta(\Gamma,y))$ such that:
  \begin{itemize}
  \item[a)] The function $v$ has constant boundary values $+\infty$ or
    $-\infty$ along $\Gamma$.
  \item[b)] Given $r>0$, we have $U_r=N_\delta(\Gamma,y)\cap B(\hat
    x,r)\cap\Omega\neq\emptyset$ and there exists $r_0>0$ such that
    $u=v$ in $U_r$ for $0<r<r_0$.
  \end{itemize}
\end{itemize}
\end{lemma}

\begin{proof}[Proof of Lemma~\ref{lem1}]
  Let us consider a geodesic ball $B(y,\rho)\subset\h^2$ contained in
  $\Omega$ and tangent to $\partial\Omega$ at $\hat x$.
  Lemma~\ref{lema:asintotico2} guarantees the existence of a curve
  $\Gamma\subset\h^2$ with constant geodesic curvature $2H$ or $-2H$,
  tangent to $B(y,\rho)$ at $\hat x$, and $\delta>0$ such that the $H$-graph over
  $B(y,\rho)$ can be extended to an $H$-graph over $B(y,\rho)\cup
  N_\delta(\Gamma,y)$. Let us define $v$ as the restriction of such an
  extension to $N_\delta(\Gamma,y)$. It can be easily shown that
  $\Gamma$, $\delta$, $y$ and $v$ satisfy the conditions in item
  ii). The uniqueness of prolongation ensures that $\Omega$ must be
  contained in $N_\infty(\Gamma,y)$, from where the estimation for the
  geodesic curvature of $\partial\Omega$ given in item i) follows.
\end{proof}

We will prove (see Lemma~\ref{lem2} for $R=R_0$ and $\mathcal
C=\emptyset$) that the set of points in $\partial B(x_0,R_0)$ such
that $u$ cannot be extended to any neighborhood of them as a vertical
$H$-graph is finite. So we can denote them as
$x_1,\ldots,x_r$. Lemma~\ref{lem1} guarantees that, given
$j\in\{1,\ldots,r\}$, there exist $\delta_j>0$ and a curve
$\Gamma_j\subset\mathbb{H}^2$ of constant geodesic curvature $2H$ or
$-2H$ which is tangent to $\partial B(x_0,R_0)$ at $x_j$, such that
$u$ can be extended to $B(x_0,R_0)\cup N_{\delta_j}(\Gamma_j,x_0)$,
for some $\delta_j>0$. The next step in the proof of
Proposition~\ref{prop} will consist of showing that the curves
$\Gamma_1,\ldots,\Gamma_r$ are disjoint. Then we will be able to
extend the initial ball $B(x_0,R_0)$ to a new admissible domain
$\Omega=B(x_0,R_0)\cup(\cup_{j=1}^rN_\delta(\Gamma_j,x_0))$, for some
$\delta\leq \min\{\delta_j :j=1,\dots, r\}$. The function associated
to such $\Omega$ will have boundary values $+\infty$ or $-\infty$
along each of the $\Gamma_j$ (depending on the sign of its geodesic
curvature). Therefore, for a further extension of $\Omega$ we will
work in $\cap_{j=1}^rN_\infty(\Gamma_j,x_0)$, i.e., we will extend
$\Omega$ in the direction of $\partial\Omega\cap\partial B(x_0,R_0)$.
As this process will be iterated, we describe a general situation for
the extension procedure.

\begin{definition}\label{defi:posicion-general}
  Let $\mathcal C$ be a finite family of curves in $\h^2$, each one
  with constant geodesic curvature $2H$ or $-2H$.
We will say that $\mathcal C$ is {\em in general position} for some
radius $R>0$ when the following three conditions are satisfied:
  \begin{enumerate}
  \item[a)] Each $\Gamma\in\mathcal C$ intersects the closed ball
    $\overline B(x_0,R)$ and $x_0\not\in\Gamma$.
  \item[b)] $\partial\Omega_{\mathcal C}\cap\Gamma\neq\emptyset$ for
    every $\Gamma\in\mathcal C$, where
    \begin{equation}\label{OmegaC}
      \Omega_{\mathcal{C}}=B(x_0,R)\cap(\cap_{\Gamma\in\mathcal C}
      N_\infty(\Gamma,x_0)).
    \end{equation}
  \item[c)] No intersection point of curves in $\mathcal C$ lies on
    $\overline B(x_0,R)$.
\end{enumerate}
\end{definition}

(We recall that $N_\infty(\Gamma,x_0)$ is the open connected component
of $\h^2\sm\Gamma$ containing $x_0$.) It is clear that the family of
curves $\{\Gamma_1,\ldots,\Gamma_r\}$ in the discussion above is in
general position for the radius $R_0$. This condition will be
preserved under the successive steps for enlarging the admissible
domain. The next lemma gives some information about a family of curves
in general position.

\begin{figure}
  \begin{center}
    \includegraphics[scale=.6]{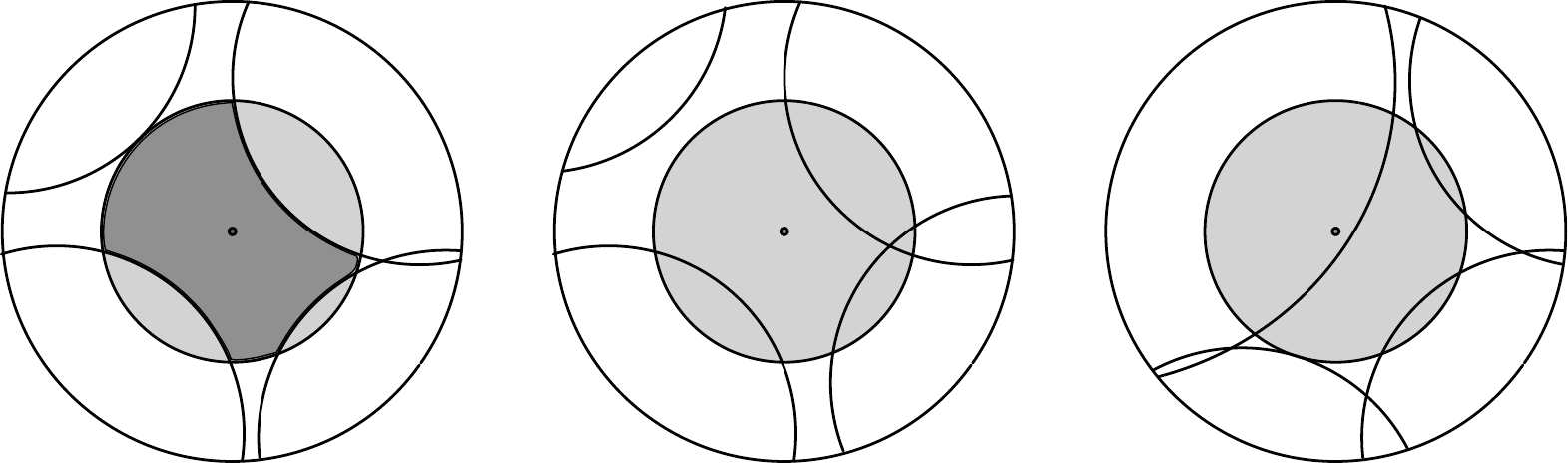}
  \end{center}
  \caption{Three families of four curves in $\h^2$. The first one is
    in general position with respect to the shaded ball and the domain
    $\Omega_{\mathcal C}$ associated to this family has been also
    shaded in a darker grey.  On the contrary, the other two families
    are not in general position: conditions (a) and (c) in the
    definition fail for the second one, whereas condition (b) does not
    hold for the third one.}\label{fig:posicion-general}
\end{figure}

\begin{lemma}\label{lem3}
  Let $\mathcal C$ be a finite family of curves in general position
  for a radius $R>0$. Suppose that $\Omega_{\mathcal C}$, defined
  by~\ref{OmegaC}, is an admissible domain and that, for any
  $\Gamma\in\mathcal C$ there exists a vertical $H$-graph $v$ over
  $N_\delta(\Gamma,x_0)$, for some $\delta>0$, with constant infinite
  boundary values along $\Gamma$, which coincides with the function
  $u$ associated to $\Omega_{\mathcal C}$ on $\Omega_{\mathcal C}\cap
  N_\delta(\Gamma,x_0)$. Then:
  \begin{enumerate}
  \item[a)] Any two curves in $\mathcal C$ are disjoint.
  \item[b)] There exists $\delta'>0$ so that $\Omega_{\mathcal
      C}\cup(\cup_{\Gamma\in\mathcal C}N_{\delta'}(\Gamma,x_0))$ is
    admissible.
  \end{enumerate}
\end{lemma}

\begin{proof}[Proof of Lemma~\ref{lem3}]
  Reasoning by contradiction, let us suppose that there exists a point
  $\widetilde x\in\Gamma_1\cap\Gamma_2$ for some
  $\Gamma_1,\Gamma_2\in\mathcal C$. Condition c) in
  Definition~\ref{defi:posicion-general} tells us that $\widetilde
  x\not\in\overline B(x_0,R)$. We can take a continuous family
  $\{D_t\}_{t\in[0,\ell)}$ of geodesic balls in $\h^2$ satisfying the
  following three conditions:
  \begin{enumerate}
  \item $D_0\subset B(x_0,R)$.
  \item $\partial D_t$ is tangent to $\Gamma_1$ and $\Gamma_2$, for
    any $t\in[0,\ell)$.
  \item The radius of $D_t$ strictly decreases with respect to $t$,
    and $D_t$ converges to the point $\widetilde x$ when $t\to\ell$.
  \end{enumerate}
  Let us call, respectively, $x_1$ and $x_2$ the points in $\Gamma_1$
  and $\Gamma_2$ which are closest to $x_0$. Denote by $T$ an open
  triangle with vertices $\widetilde x$, $x_1$ and
  $x_2$, which has two sides lying on $\Gamma_1$ and $\Gamma_2$ and the third one is a curve interior to $\Omega_{\mathcal C}$ joining $x_1$ and $x_2$.
  Let us denote by $\Lambda_t\subset \partial D_t$ the intersection of
  $\Omega_{\mathcal C}\cup T$ with the longest of the two curves in
  which $\partial D_t$ is divided by $\Gamma_1$ and $\Gamma_2$. Then
  $\Omega_{\mathcal C}\cup T=\Omega_{\mathcal C}\cup (\cup_{[0,\ell)}
  \Lambda_t)$ (see figure~\ref{fig:barrer-triangulo}).

  \begin{figure}
    \centering\includegraphics[scale=.6]{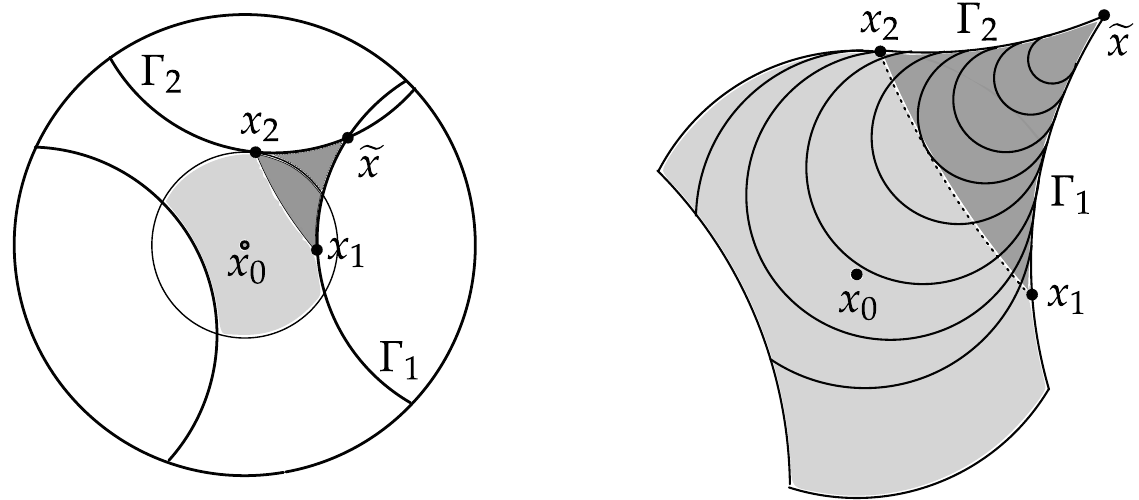}
    \caption{Left: Three curves in general position, two
      of which intersect at $\widetilde x\not\in\overline
      B(x_0,R)$. Right:
      $\Omega_{\mathcal C}\cup T$ have been magnified and some of the
      curves $\Lambda_t$ which cover $T$ up to $\widetilde x$ are
      shown.}
    \label{fig:barrer-triangulo}
  \end{figure}

  Let $r\in[0,\ell]$ be the supreme of the values of $t\in[0,\ell)$
  for which the function $u$ associated to $\Omega_{\mathcal C}$ can
  be extended as an $H$-graph on $\Omega_{\mathcal
    C}\cup(\cup_{s\in[0,t]}\Lambda_s)$. Note that $r>0$ since the
  graph can be extended to $\Omega_{\mathcal C}\cup
  N_\varepsilon(\Gamma_i,x_0)$, for $i\in\{1,2\}$ and some
  $\varepsilon>0$, and the curves $\Lambda_s$ lie on the union of
  these two open sets for small values of $s$. Moreover, it must be
  $r=\ell$: Otherwise there would exist a point in $\Lambda_r$ such
  that $u$ cannot be extended to any neighborhood of it as an
  $H$-graph. But $\Lambda_r$ has constant geodesic curvature smaller
  than $-1<-2H$ with respect to the normal pointing to
  $\Omega_{\mathcal C}\cup(\cup_{s\in[0,r]}\Lambda_s)$, contradicting
  Lemma~\ref{lem1}.

  Hence we have proved that $\Omega_1=\Omega_{\mathcal C}\cup T$ is
  admissible, and the extension $u_1$ of $u$ over $\Omega_1$ has
  non-bounded values in the part of $\partial\Omega_1$ lying on
  $\Gamma_1\cup\Gamma_2$. Nonetheless, the graph $u$ can also be
  extended to $u_2$ defined over $\Omega_2=\Omega_{\mathcal C}\cup
  N_\varepsilon(\Gamma_1,x_0)$. Thus $u_1=u_2$ in
  $\Omega_1\cap\Omega_2$, by uniqueness of the analytic
  prolongation. This is a contradiction because $u_2$ has bounded
  values in $\Gamma_2\cap
  N_\varepsilon(\Gamma_1,x_0)\subset\partial(\Omega_1\cap\Omega_2)$
  whereas $u_1$ does not. Such a contradiction proves a).  Item b)
  also follows from the previous argument.
\end{proof}

Let us now prove that, apart from the curves along which the graph
takes infinite boundary values, the number of points which do not
admit an extension to any neighborhood of them is finite at any step
of the extension procedure.

\begin{lemma}\label{lem2}
  Let $\mathcal C$ be a finite family of curves in general position
  for some radius $R>0$ and suppose that $\Omega_{\mathcal C}$ is
  admissible. Then the set $A(R)$ of points in
  $\partial\Omega_{\mathcal C}\cap\partial B(x_0,R)$ such that the
  function $u$ associated to $\Omega_{\mathcal C}$ cannot be extended
  to any neighborhood of them as an $H$-graph, is finite.
\end{lemma}

\begin{proof}[Proof of Lemma~\ref{lem2}]
  Since $\mathcal C$ is finite, $\partial\Omega_{\mathcal
    C}\cap\partial B(x_0,R)$ consists of finitely many regular
  arcs. Their endpoints belong to $A(R)$, but they are finitely many
  and, by Lemma~\ref{lem1}, it is known that the graph can be extended
  as an $H$-graph to a neighborhood in $\h^2$ of any point in a
  neighborhood in $\partial\Omega_{\mathcal C}\cap\partial B(x_0,R)$
  of any of them. Thus we will not consider such endpoints.

  Lemma~\ref{lem1} also says that the rest of points in $A(R)$ are
  isolated. On the other hand, it is easy to check that $A(R)$ is
  closed: We suppose there exists a sequence $\{x_n\}$ in $A(R)$
  converging to $x_\infty\not\in A(R)$. Such a point
  $x_\infty$ must be interior to one of the arcs in $\partial\Omega_{\mathcal
    C}\cap\partial B(x_0,R)$ because of the discussion above. Then $u$
  admits an extension to a neighborhood of $x_\infty$ in $\h^2$. But
  $x_n$ lies in such a neighborhood for $n$ large enough, a
  contradiction.

  Hence $A(R)$ is closed, consists of isolated points and is contained
  in the compact set $\partial B(x_0,R)$, so it is finite.
\end{proof}

We now have all the ingredients to prove the desired result.

\begin{proof}[Proof of Proposition~\ref{prop}]
  Repeating the argument leading to
  equation~\eqref{eqn:multigrafo-radio-inicial} and provided that
  $R_0<+\infty$ (otherwise we would be done), we can consider the
  maximal admissible ball $B(x_0,R_0)$. Lemma~\ref{lem2} guarantees
  the existence of a finite collection of points in $\partial
  B(x_0,R_0)$ such that $B(x_0,R_0)$ cannot be extended in an
  admissible way to any neighborhood of any of them. Hence, there
  exists a finite family $\mathcal C_0$ of complete curves with
  constant geodesic curvature $2H$ or $-2H$, each one tangent to
  $\partial B(x_0,R_0)$ at one of those points, under the conditions
  of Lemma~\ref{lem1}. The family $\mathcal C_0$ is in general
  position. Then Lemma~\ref{lem3} says that the curves in $\mathcal
  C_0$ are disjoint and that there exists $\delta>0$ such that the
  domain $\Omega'=\Omega_{\mathcal C_0}\cup(\cup_{\Gamma\in\mathcal
    C_0}N_\delta(\Gamma,x_0))$ is admissible.

  Moreover, the initial graph can be extended to a neighborhood of any
  point interior to $\partial\Omega'\cap\partial B(x_0,R_0)$, so we
  can define the supremum of $R>R_0$ satisfying that
  $B(x_0,R)\cap(\cap_{\Gamma\in\mathcal C_0} N_\infty(\Gamma,x_0))$ is
  admissible, called $R_1$. If $R_1<+\infty$, then let us consider the
  admissible domain $\Omega_1=B(x_0,R_1)\cap(\cap_{\Gamma\in\mathcal
    C_0} N_\infty(\Gamma,x_0))$. By definition of $R_1$, there will be
  points in the interior of $\p\Omega_1\cap\p B(x_0,R_1)$ for which
  $\Omega_1$ cannot be extended in an admissible way to any
  neighborhood of any of them. By Lemma~\ref{lem2}, these points are
  finitely many, and we can apply Lemma~\ref{lem1} to obtain a new
  family $\mathcal C_1\supset\mathcal C_0$ in general position for the
  radius $R_1$. Note that $\mathcal C_1-\mathcal C_0\neq \emptyset$.

  This procedure can be iterated to get a strictly increasing sequence
  (possibly finite) or radii $\{R_n\}$ in such a way that, for any
  $n$, there exists $\mathcal C_n$, a family of curves in general
  position for $R_n$, with $\mathcal C_{n-1}\subsetneq\mathcal
  C_n$. Besides, the arguments above show that the domain
  \[
  \Omega_{\mathcal C_n}=B(x_0,R_n)\cap\left(\bigcap_{\Gamma\in\mathcal
      C_n}N_\infty(\Gamma,x_0)\right),
  \]
  is admissible, $\Omega_{\mathcal C_{n-1}}\subset\Omega_{\mathcal
    C_n}$ and the associate function $u_n$ extends $u_{n-1}$. Note
  that there exists $\delta_n>0$ such that $\Omega_{\mathcal C_n}$ can
  be extended to the admissible domain $\Omega_{\mathcal
    C_n}\cup(\cup_{\Gamma\in\mathcal C_n}N_{\delta_n}(\Gamma,x_0))$,
  by item b) of Lemma~\ref{lem3}.

  In this situation, there are two possibilities: Either
  $R_{n_0}=+\infty$ for some $n_0\geq 0$ (and we would be done since
  $\Sigma$ would be a complete vertical graph), or the sequence
  $\{R_n\}$ has infinitely many terms. Suppose we are in the latter
  case. Then we claim that $\lim\{R_n\}=+\infty$. In order to prove
  this, we observe that the curves in $\mathcal C_n$ are all disjoint
  by Lemma~\ref{lem3}. If $R_\infty=\lim\{R_n\}<+\infty$, we would
  have an infinite family of disjoint curves (infinite, as each
  $\mathcal C_n$ strictly contains $\mathcal C_{n-1}$), each one with
  constant geodesic curvature $2H$ or $-2H$ and intersecting the
  compact ball $\overline B(x_0,R_\infty)$. But this situation is
  impossible by condition b) in Definition~\ref{defi:posicion-general}
  (it tells us that, when we fix one of such curves, $\Gamma$, the
  rest of them must lie in one of the two components of
  $\h^2\sm\Gamma$).

  Given $n\in\N$, the open set $O_n=F_{u_n}(\Omega_{\mathcal
    C_n})\subset\Sigma$ satisfies that $\pi(\partial
  O_n)\subset\partial B(x_0,R_n)$, since $u_n$ has boundary values
  $\pm\infty$ on the curves of $\mathcal C_n$. Then we get that the
  length of any piecewise regular curve $\alpha:[a,b]\to\Sigma$ with
  $\alpha(a)=p$ and $\alpha(b)\in\partial O_n$ is bigger than $R_n$,
  as the projected curve $\pi\circ\alpha$ in $\h^2$ is shorter than
  $\alpha$. Thus, the distance in $\Sigma$ from $p$ to $\partial O_n$
  is at least $R_n$, so $B_\Sigma(p,R_n)\subset F_u(\Omega_{\mathcal
    C_n})$ and we are done since $R_n$ is arbitrarily large.
\end{proof}

\end{document}